\documentclass[11pt]{article}
\usepackage{url}
\usepackage{amsfonts}
\usepackage{amsthm}
\usepackage{amsmath}
\usepackage{amssymb}
\usepackage{tikz}
\usepackage{tikz}
\usepackage{longtable}
\usepackage{authblk}
\usepackage{array}
\newcolumntype{x}[1]{>{\centering\arraybackslash\hspace{0pt}}p{#1}}
\usepackage[font=normalsize,labelfont={bf}]{caption}
\usepackage[top=1.25in, bottom=1.5in, left=1.5in, right=1.5in]{geometry}%
\newtheorem{theorem}{Theorem}[section]
\newtheorem{definition}{Definition}[section]
\newtheorem{example}{Example}[section]

\newtheorem{lemma}[theorem]{Lemma}
\newtheorem{corollary}[theorem]{Corollary}	

\usepackage[flushleft]{threeparttable}

\newcommand{\zed}{{\ensuremath{\mathbb{Z}}}}

\newcommand{\eff}{{\ensuremath{\mathbb{F}}}}

\begin{document}

\title{Some new results on skew 
frame starters in cyclic groups
}

\author[1,2]{Douglas R.\ Stinson\thanks{The author's research is supported by  NSERC discovery grant RGPIN-03882.}}
\affil[1]{School of Mathematics and Statistics\\
Carleton University\\
Ottawa, Ontario, K1S 5B6, Canada}
\affil[2]{David R.\ Cheriton School of Computer Science\\University of Waterloo\\ Waterloo ON, N2L 3G1\\Canada}

\date{\today}

\maketitle

\begin{abstract}
In this paper, we study skew frame starters, which are strong frame starters that satisfy an additional ``skew'' property. 
We prove three new non-existence results for cyclic skew frame starters of certain types. We also construct several small examples of
previously unknown cyclic skew frame starters by computer.
\end{abstract}

\section{Introduction}

\label{fs.sec}

The paper \cite{St22} is a recent work discussing existence of cyclic strong frame starters. Here, we turn our attention to cyclic skew frame starters. We begin with relevant definitions and a review of results on cyclic skew starters in Section \ref{fs.sec}. We observe that there is a gap in the proof of a well-known existence theorem (namely, Theorem \ref{general.thm})
from \cite{CGZ} and we show how to complete the missing details in the proof.
We prove three nonexistence theorems in Section \ref{nonexist.sec}. The first nonexistence result uses a classical method previously employed by Constable \cite{Co74} and Wallis and Mullin \cite{WM}. The other two nonexistence results are obtained by analyzing properties of  homomorphic images of hypothetical cyclic skew frame starters.  In Section \ref{small.sec}, we discuss existence and nonexistence of ``small''
cyclic skew frame starters, and several previously unknown cyclic skew frame starters are constructed by computer. Section \ref{summary.sec} is a brief summary and discussion.

We recall some standard definitions now. 

\begin{definition}
Let $G$ be an additive abelian group of order $g$ and let $H$ be a subgroup of $G$ of order $h$, where $g-h$ is even. A \emph{frame starter} in $G \setminus H$ is a set of $(g-h)/2$ pairs $\{ \{x_i,y_i \} : 1 \leq i \leq (g-h)/2\}$ that satisfies the following two properties:
\begin{enumerate}
\item $\{ x_i, y_i :  1 \leq i \leq (g-h)/2 \} = G \setminus H$.
\item $\{ \pm (x_i-y_i) :  1 \leq i \leq (g-h)/2 \} = G \setminus H$.
\end{enumerate}
This frame starter has \emph{type} $h^{g/h}$.
\end{definition}
Note that the pairs in the frame starter form a partition of $G \setminus H$, and the differences obtained from these pairs also partitions $G \setminus H$.  If $H = \{0\}$, then the frame starter is just called a \emph{starter}.

\begin{definition}
Suppose that $S = \{ \{x_i,y_i \} : 1 \leq i \leq (g-h)/2\}$ is a frame starter in $G \setminus H$. $S$  is \emph{strong}
 if the following two properties hold:
\begin{enumerate}
\item $x_i + y_i \not\in H$ for $1 \leq i \leq (g-h)/2$.
\item $x_i + y_i  \neq x_j + y_j$ if $1 \leq i,j \leq (g-h)/2$, $i \neq j$. 
\end{enumerate}
\end{definition}

\begin{definition}
Suppose that $S_1 = \{ \{x_i,y_i \} : 1 \leq i \leq (g-h)/2\}$ and $S_2 = \{ \{u_i,v_i \} : 1 \leq i \leq (g-h)/2\}$
are both frame starters in $G \setminus H$. Without loss of generality, assume that
$y_i - x_i = v_i - u_i$ for $1 \leq i \leq (g-h)/2$.  $S_1$ and $S_2$ are
\emph{orthogonal} if the following two properties hold:
\begin{enumerate}
\item $y_i - v_i \not\in H$ for $1 \leq i \leq (g-h)/2$.
\item $y_i - v_i \neq y_j - v_j$ if $1 \leq i,j \leq (g-h)/2$, $i \neq j$. 
\end{enumerate}
\end{definition}
In other words, when the pairs in $S_1$ and $S_2$ are matched according to their differences, the ``translates'' are distinct elements of $G \setminus H$. These translates are often called the \emph{adder}.

It is not hard to see that a strong frame starter $S$ is orthogonal to $-S$. The associated adder is $a_i = -(x_i + y_i)$, $1 \leq i \leq (g-h)/2$.

\begin{definition}
Suppose that $S = \{ \{x_i,y_i \} : 1 \leq i \leq (g-h)/2\}$ is a frame starter in $G \setminus H$. $S$  is \emph{skew}
 if 
 \[ \{\pm (x_i + y_i) :  1 \leq i \leq (g-h)/2\}  = G \setminus H.\]
 It is clear that any skew starter is also strong.
\end{definition}

The skew property is aesthetically pleasing. However, we should note that skew starters (especially in cyclic groups) have some unexpected applications to the construction of other types of designs. See \cite{CGZ,DS92} for some examples.

\begin{example}
\label{zed10}
Suppose $G =\zed_{10}$ and $H = \{0,5\}$. Here is a skew frame starter of type $2^5$ in $G\setminus H$:
\[
\begin{array}{l}
S = \{ \{3,4\}, \{7,9\}, \{8,1\}, \{2,6\} \}. 
\end{array}
\]
\end{example}

\begin{example}
Suppose $G =\zed_{7}$ and $H = \{0\}$. Here is a skew frame starter of type $1^7$ (i.e., a skew starter) in $G\setminus H$:
\[
\begin{array}{l}
S = \{ \{2,3\}, \{5,1\}, \{6,4\} \}. 
\end{array}
\]
\end{example}

\begin{example}
\cite{St80}
\label{z4z4.exam}
Suppose $G =\zed_{4} \times \zed_{4}$ and $H = \{(0,0), (0,2), (2,0), (2,2)\}$. Here is  a skew frame starter of type $4^4$ in $G\setminus H$:
\[
\begin{array}{ll}
S = &\{ \{(1,1),(3,2)\}, \{(3,0),(3,1)\}, \{(2,1),(3,3)\}, \{(0,3),(1,3)\},\\
&   \{(1,0),(2,3)\}, \{(0,1),(1,2)\} \}. 
\end{array}
\]
\end{example}


A frame starter $S$ is \emph{patterned} if $S = \{ \{x,-x\} : x \in G \setminus H\}$.
There is a (unique) patterned frame starter in $G \setminus H$ whenever $G$ has odd order.

The next theorem generalizes a  classical result due to Byleen \cite{By70} to the setting of frame starters.

\begin{theorem} 
\label{patterned.thm}
Suppose $G$ is an abelian group of odd order. Then there is a strong frame starter $S$ in $G \setminus H$ if and only if there is an adder $A$ for the patterned frame starter in $G \setminus H$. Further, $S$ is a skew frame starter if and only if the adder $A$ is skew. 
\end{theorem} 

\begin{proof} Suppose $S$ is a patterned frame starter in $G \setminus H$. Suppose that $a_i$ is the adder associated with a pair $\{x_i, -x_i\}$. The corresponding orthogonal frame starter $T$ contains the pair $\{x_i + a_i, -x_i + a_i\}$. We just need to check that $T$ is strong. The sums of the pairs in $T$
are $2a_i$, $1 \leq i \leq (g-h)/2$. These sums are clearly distinct, so we need only prove that $2a_i \not\in H$ for all $i$. But this follows easily from the observation that the mapping $x \mapsto 2x$ is a bijection of $G$ that maps $H$ to $H$. 

Conversely, suppose that $S$ is a strong frame starter. For any pair $\{x_i,y_i\} \in S$, let $a_i =  (x_i + y_i)/2$ be the adder for the pair
$\{s_i,t_i\}$ in the patterned frame starter, where $s_i = (x_i-y_i)/2$ and $t_i = (y_i-x_i)/2$. We check that $s_i + a_i = x_i$ and $t_i + a_i = y_i$, so $S$ is orthogonal to the patterned frame starter. It is easy to see that no element $a_i \in H$, again using the fact that the mapping $x \mapsto 2x$ is a bijection of $G$ that maps $H$ to $H$. 

Finally, $S$ is skew if and only if $\{\pm (x_i + y_i) :  1 \leq i \leq (g-h)/2\}  = G \setminus H$, and $A$ is skew if and only if
$\{\pm (x_i + y_i)/2 :  1 \leq i \leq (g-h)/2\}  = G \setminus H$. Hence, $S$ is skew if and only if $A$ is skew.
\end{proof}


The main topic we study in this paper is skew frame starters. However, there has been considerable prior work on the special case of skew starters. For completeness, we recall some known existence results. The following infinite class of skew starters is usually called the  \emph{Mullin-Nemeth starters}.

\begin{theorem}\cite{MN}\label{MN.thm}
Suppose $q= 2^kt+1$ is a prime power where $t> 1$ is odd. Then there is a skew starter in $\eff_q \setminus \{0\}$. 
\end{theorem}

There is another infinite class of skew starters that is often called the \emph{Chong-Chan-Dinitz starters}.  
We refer to Lins and Schellenberg \cite{LS} for a simple presentation of this class of starters.

\begin{theorem}\cite{LS}\label{LS.thm}
Suppose that $n = 16t^2 + 1$. Then there is a skew starter in $\zed_n$.
\end{theorem}

The following consequence of Theorems \ref{MN.thm} and \ref{LS.thm} is well-known.

\begin{corollary}
\label{prime.cor}
If $n > 5$ is prime, then there is a skew starter in $\zed_n$.
\end{corollary}

\begin{proof}
Write $n= 2^kt+1$, where $t\geq 1$ is odd. If $t> 1$, then there is a Mullin-Nemeth skew starter in $\zed_n$. If $t = 1$, then
$n = 2^k + 1$ is prime. Therefore $n$ must be a Fermat prime, i.e., $n = 2^{2^s}+1$ for some integer $s \geq 0$. Since we assumed that
$n > 5$, we have $s \geq 2$, and therefore $n = 16t^2 + 1$ for a positive integer $t$. Then the Chong-Chan-Dinitz starter is a skew starter in $\zed_n$.
\end{proof}

Skew starters in $\zed_n$ have been studied in detail by Liaw \cite{Li} and by Chen, Ge and Zhu \cite{CGZ}.
The most general current existence result for skew starters is in \cite{CGZ}.

\begin{theorem}
\label{general.thm} 
\cite[Theorem 1.2]{CGZ}
Suppose that $\gcd(n,6) = 1$ and either $n \not\equiv 0 \bmod 5$ or $n \equiv 0 \bmod 125$. Then there is a skew starter in $\zed_n$.
\end{theorem}

It seems that there is a gap in the proof of Theorem \ref{general.thm} given in \cite{CGZ}. However, this gap is easily filled using techniques from \cite{CGZ}.
We discuss this now.

Dinitz and Stinson in \cite[p.\ 191]{DS92} posed the following open problem: ``Prove that there exists a skew starter in $\zed_n$ for all $n$ such that $\gcd(n,6) = 1$.''\footnote{Of course the condition $n > 5$ also needs to be included. Otherwise, the conjecture is false because there is no skew starter in $\zed_5$.}
This open problem was recalled in \cite{CGZ} and Theorem \ref{general.thm} was presented as a partial answer to the problem. The proof of Theorem \ref{general.thm} given in \cite{CGZ} uses the following lemma as a starting point.

\begin{lemma}
\label{30.lem}
If $\gcd(n,30) = 1$, then there is a skew starter in $\zed_n$.
\end{lemma}

In \cite{CGZ}, Lemma \ref{30.lem} is attributed to Dinitz and Stinson \cite{DS92}. However, this lemma does not appear in the cited article \cite{DS92} and neither Dinitz nor Stinson can recall having seen this lemma at the time that \cite{DS92} was published. Fortunately, the techniques 
of \cite{CGZ} suffice to prove Lemma \ref{30.lem}. The main tool is the following powerful multiplication construction from \cite{CGZ}.

\begin{lemma}\cite[Lemma 2.3]{CGZ}
\label{mult.lem}
Suppose there exist skew starters in $\zed_m$ and $\zed_n$. Then there is a skew starter in $\zed_{mn}$.
\end{lemma}

Now, it is not hard to see that Lemma \ref{30.lem} is a consequence of Corollary \ref{prime.cor} and Lemma \ref{mult.lem}. 
First, for any prime $p > 5$ and any integer $s \geq 1$, Corollary \ref{prime.cor} and Lemma \ref{mult.lem} prove that there is a skew starter in $\zed_{p^s}$. Then additional applications of Lemma \ref{mult.lem} can be used to handle all $\zed_n$ where $\gcd(n,30)= 1$, since these values of $n$ have no prime divisors $\leq 5$.

Thus it is not difficult to reconstruct a proof of Lemma \ref{30.lem}; then the rest of the proof of Theorem \ref{general.thm} from \cite{CGZ} goes through unchanged.

\bigskip

There are very few prior existence results on skew frame starters. However, here is one infinite class that is known to exist.

\begin{theorem}\cite{DS,SW}
\label{DS,SW.thm}
Suppose $q\equiv 1 \bmod 4$ is a prime power and $n \geq 1$. Then there is a skew frame starter in $(\eff_q \times (\zed_2)^n)  \setminus (\{0\}\times (\zed_2)^n)$. 
\end{theorem}

Finally, it will be useful to recall some nonexistence results for frame starters and strong frame starters. These results (due to various authors) can all be found in \cite{St22}.

\begin{theorem}
\label{nonexist2.thm}
Suppose $G$ is an abelian group of order $2u$ and suppose $H$ is a subgroup of $G$ of order $2t$, where $t$ is odd. If  $u/t \equiv 2\text{ or }3 \bmod 4$, then there is no frame starter in $G \setminus H$.
\end{theorem}

\begin{theorem}
\label{t^5strong.thm}
Suppose $t$ is odd, $G$ is an abelian group of order $5t$ and $H$ is a subgroup of order $t$. Then there does not exist a strong frame starter in $G \setminus H$.
\end{theorem}

\begin{theorem}
\label{newnonexist.thm}
Suppose $G$ is an abelian group of order $4t$ and suppose $H$ is a subgroup of $G$ of order $t$, where $t$ is even and $G / H \cong \zed_4 $. Then there is no strong frame starter in $G \setminus H$.
\end{theorem}

\begin{theorem}
\label{newnonexist2.thm}
Suppose $G$ is an abelian group of order $6t$ and suppose $H$ is a subgroup of $G$ of order $t$. 
Then there is no strong frame starter in $G \setminus H$.
\end{theorem}


\section{Three new nonexistence theorems}
\label{nonexist.sec}

In the rest of this paper, we restrict our attention to skew frame starters in $G \setminus H$, where $G = \zed_g$
and $H =  \{0, r, 2r, \dots , (h-1)r\}$ (note that $H$ is the unique subgroup of $G$ having order $h$). 
We will refer to such a starter as being a \emph{cyclic} skew frame starter of type $h^{g/h}$.

First we consider odd order cyclic groups.  The following result uses a technique due to Constable \cite{Co74} and Wallis and Mullin \cite{WM}.

\begin{lemma}
\label{zerosum.lem}
Suppose $S = \{ \{x_i,y_i \} : 1 \leq i \leq (g-h)/2\}$ is a cyclic strong frame starter of type $h^{g/h}$, where $g$ is odd. Then
\begin{equation}
\sum_{i=1} ^{(g-h)/2} \left( {x_i} + {y_i} \right)^2 \equiv 0 \bmod g.
\end{equation}
\end{lemma}

\begin{proof} Since $S$ is strong and $g$ is odd, it follows from Theorem \ref{patterned.thm} that the values $a_i = ({x_i} + {y_i})/2$ comprise an adder for the patterned frame starter in $G \setminus H$  ($a_i$ is the adder element for the pair 
$\{(x_i-y_i)/2, (y_i-x_i)/2\}$. Let $\{s_i,t_i\}$ denote the pairs in the patterned frame starter, where $s_i = (x_i-y_i)/2$ and $t_i = (y_i-x_i)/2$.
We have 
\[ \{ s_i: 1 \leq i \leq (g-h)/2 \} \cup \{t_i: 1 \leq i \leq (g-h)/2\} = G \setminus H\] and 
\[ \{ s_i + a_i: 1 \leq i \leq (g-h)/2\} \cup \{t_i + a_i: 1 \leq i \leq (g-h)/2\} = G \setminus H.\]
Hence,
\[\{ s_i,t_i : 1 \leq i \leq (g-h)/2\}= 
\{ s_i + a_i, t_i + a_i: 1 \leq i \leq (g-h)/2\}, \]
and therefore we have 
\[\sum_{i=1} ^{(g-h)/2} \left( {s_i}^2 + {t_i}^2 \right) \equiv 
\sum_{i=1} ^{(g-h)/2}  \left( (s_i+a_i)^2 + (t_i+a_i)^2 \right) \bmod g.\]
Expanding the right hand side, we have
\[\sum_{i=1} ^{(g-h)/2} \left( {s_i}^2 + {t_i}^2 \right) \equiv 
\sum_{i=1} ^{(g-h)/2}  \left( {s_i}^2 + 2s_ia_i + {a_i}^2 + {t_i}^2 + 2t_ia_i + {a_i}^2 \right) \bmod g,\]
 so
\[
\sum_{i=1} ^{(g-h)/2}  \left( 2s_ia_i + {a_i}^2 +  2t_ia_i + {a_i}^2 \right) \equiv 0 \bmod g.\]
Hence,
\[
\sum_{i=1} ^{(g-h)/2}  \left( 2a_i(s_i + t_i) + 2{a_i}^2  \right) \equiv 0 \bmod g.\]
Since $s_i + t_i \equiv 0 \bmod g$ for all $i$, it follows that
\[
\sum_{i=1} ^{(g-h)/2} 2{a_i}^2 \equiv 0 \bmod g.\]
Now, substitute $a_i = ({x_i} + {y_i})/2$ to obtain
\[
\sum_{i=1} ^{(g-h)/2} ({x_i} + {y_i})^2/2 \equiv 0 \bmod g.\]
Finally, since $g$ is odd, we have
\[
\sum_{i=1} ^{(g-h)/2} ({x_i} + {y_i})^2 \equiv 0 \bmod g.\]
\end{proof}

Suppose $g = hr$ is odd. As before, let $G = \zed_g$ and let
$H =  \{0, r, 2r, \dots , (h-1)r\}$ be the unique subgroup of $G$ of order $h$.
Define \[ \mathsf{half}(G,H) = \{ j : 1 \leq j \leq (g-1)/2, \, j \not\equiv 0 \bmod r\} .\]

\begin{lemma}
\label{zerosumodd.lem}
Let $g = hr$ be odd.
Suppose $S = \{ \{x_i,y_i \} : 1 \leq i \leq (g-h)/2\}$ is a cyclic skew frame starter of type $h^{g/h} = h^r$.
\begin{equation}
\sum _{j \in \mathsf{half}(G,H) } j^2 \equiv 0 \bmod g.
\end{equation}
\end{lemma}

\begin{proof}
We are assuming that $S$ is a skew frame starter in $G \setminus H$, where $G = \zed_g$
and $H =  \{0, r, 2r, \dots , (h-1)r\}$.
We first observe that 
\[ \mathsf{half}(G,H) \cup (- \mathsf{half}(G,H)) = G \setminus H.\]
For every $j \in G \setminus H$, there is a unique pair $\{x_i,y_i\} \in S$ such that $x_i + y_i \equiv \pm j \bmod g$
(this is because $S$ is skew). 
Therefore
\[ \sum _{j \in \mathsf{half}(G,H) } j^2 \equiv \sum _{i = 1 }^{(g-h)/2} (x_i + y_i)^2  \bmod g.\]
The result then follows immediately from Lemma \ref{zerosum.lem}.
\end{proof}

\begin{lemma}
\label{sum.lem}
Let $g = hr$ be odd. Then 
\begin{equation}
\label{sum.eq}
\sum _{j \in G \setminus H} j^2 = \frac{g(2gh-1)(g-h)}{6h}.
\end{equation}
\end{lemma}

\begin{proof}
It is clear that 
\[\sum _{j \in G \setminus H} j^2 = \sum_{i=1} ^{g-1} i^2 - r^2\sum_{i=1} ^{h-1} i^2.\]
Using the standard formula 
\[\sum _{j =1}^n j^2  = \frac{n(n+1)(2n+1)}{6}\]
twice, and simplifying, the stated result is obtained.
\end{proof}

\begin{theorem}
\label{main.thm}
Let $g$ be odd and suppose $h$ is a divisor of $g$. Suppose that
\begin{equation}
\label{cong.eq}
(2gh-1)(g-h) \not\equiv 0 \bmod 6h.
\end{equation} Then there is no cyclic skew frame starter of type $h^{g/h}$ (in $\zed_g$). 
\end{theorem}

\begin{proof}
Suppose there is a skew frame starter in $G \setminus H$, where $G = \zed_g$
and $H =  \{0, r, 2r, \dots , (h-1)r\}$. Let \[T = \sum _{j \in \mathsf{half}(G,H) } j^2.\] 
From Lemma \ref{zerosumodd.lem}, we have $T \equiv 0 \bmod g$. Since $g$ is odd,  
$T \equiv 0 \bmod g $ if and only if $2T \equiv 0 \bmod g$. However,
\begin{eqnarray*}
 2T  & \equiv & \sum _{j \in G \setminus H } j^2 \bmod g \\
 &\equiv & \frac{g(2gh-1)(g-h)}{6h} \bmod g 
 \end{eqnarray*}
from Lemma \ref{sum.lem}.

So $T \equiv 0 \bmod g$ if and only if ${(2gh-1)(g-h)}/{6h}$ is an integer. Therefore, if 
$(2gh-1)(g-h) \not\equiv 0 \bmod 6h$, 
it follows that $T \not\equiv 0 \bmod g$ and hence a  cyclic skew frame starter of type $h^{g/h}$ cannot exist.
\end{proof}

We look at some consequences of Theorem \ref{main.thm}. First, we consider the case $h=1$. 

\begin{corollary}
\label{h=1.thm}
If $t$ is odd, then there does not exist a cyclic skew  starter  in $\zed_{3t}$.
\end{corollary}

\begin{proof}
When $h=1$, the condition (\ref{cong.eq}) reduces to 
\[(2g-1)(g-1) \not\equiv 0 \bmod 6.\] Recalling that $g$ is odd, this condition holds if and only if $g \equiv 3 \bmod 6$.
If $g \equiv 3 \bmod 6$, then $g = 3t$ where $t$ is odd.
Hence, there is no cyclic skew  starter  in $\zed_{3t}$ when $t$ is odd. 
\end{proof}
 
We note that Corollary \ref{h=1.thm} is a classic result proven in \cite{Co74,WM}.\footnote{More precisely, Constable \cite{Co74} proved Corollary \ref{h=1.thm}. Subsequently, Wallis and Mullin proved a generalization that applies to any abelian group having order divisible by three in which the $3$-Sylow subgroup is cyclic.}
The next case is $h=3$.  

\begin{corollary}
\label{h=3.thm}
If $t \equiv 3 \text{ or } 5 \bmod 6$, then there does not exist a cyclic skew frame starter of type $3^{t}$ (in $\zed_{3t})$.
\end{corollary}

\begin{proof}
When $h=3$, the condition (\ref{cong.eq})  becomes 
\[(6g-1)(g-3) \not\equiv 0 \bmod 18.\]
Since $\gcd(6g-1,18) = 1$, this condition reduces to  $g \not\equiv 3 \bmod 18$. Writing $g = 3t$, the result follows from Theorem \ref{main.thm}.
\end{proof}

The following corollary is proven in a similar manner as Theorem \ref{h=3.thm}.

\begin{corollary}
\label{h=5.thm}
If $t \equiv 3 \bmod 6$, then there does not exist a cyclic skew frame starter of type $5^{t}$ in $\zed_{5t}$.
\end{corollary}

For even order cyclic groups, we use a different approach. Some nonexistence results for strong frame starters have previously been obtained by considering a homomorphic image of a putative strong frame starter (see, e.g., \cite{St22}). Here we assume the existence of a cyclic skew frame starter $S$ of type 
$h^{3t}$ in $\zed_{g}$, where $g = 3ht$, and consider the image of $S$ under the canonical homomorphism $\phi : \zed_g \rightarrow \zed_3$.
Note that $g$ can be even or odd in this analysis. 

\begin{theorem}
\label{three.thm}
There does not exist a cyclic skew strong starter of type 
$h^{3t}$ in $\zed_{3ht}$ if $ht \not\equiv 0 \bmod 3$.
\end{theorem}

\begin{proof}
Denote $g = 3ht$. 
Let $\phi : \zed_{g} \rightarrow \zed_3$ be the homomorphism defined as $x \mapsto x \bmod 3$.
Suppose $S = \{ \{x_i,y_i \} : 1 \leq i \leq (g-h)/2\}$ is a skew frame starter in $G \setminus H$, where $G = \zed_g$ and $H$ is the subgroup of $G$ having order $h$. 
Define the \emph{type} of an element $x \in G$ to be $\phi(x)$, and define the type of a pair $\{x,y\} \in S$ to be the multiset
$\{ \phi(x), \phi(y)\}$. For $ 0 \leq i \leq j \leq 2$, 
let $a_{\{i,j\}}$ denote the number of pairs in $S$ of type $\{i,j\}$.

$G \setminus H$ contains $g/3 - h$ elements of type $0$, $g/3$ elements of type $1$ and $g/3$ elements of type $2$. 
Thus the following three equations are obtained:
\begin{eqnarray}
\label{e1}2a_{\{0,0\}} + a_{\{0,1\}} + a_{\{0,2\}}   & = & g/3 - h\\
\label{e2}2a_{\{1,1\}} + a_{\{0,1\}} + a_{\{1,2\}}   & = & g/3\\
\label{e3}2a_{\{2,2\}} + a_{\{0,2\}} + a_{\{1,2\}}    & = & g/3.
\end{eqnarray}

There are $\frac{1}{2}(\frac{g}{3} - h)$ pairs in $S$ having a difference of type $0$ and
there are $g/3$ pairs in $S$ that have a difference of type $1$ or $2$. So we obtain two further equations:
\begin{eqnarray}
\label{e4} a_{\{0,0\}} + a_{\{1,1\}} + a_{\{2,2\}}   & = &  (g/3 - h)/2\\
\label{e5} a_{\{0,1\}} + a_{\{0,2\}} + a_{\{1,2\}}   & = & g/3.
\end{eqnarray}

Finally, there are $\frac{1}{2}(\frac{g}{3} - h)$ pairs in $S$ whose sum has type $0$; these are the pairs of type $\{0,0\}$ or $\{1,2\}$. So we obtain the following equation:
\begin{eqnarray}
\label{e6}a_{\{0,0\}} + a_{\{1,2\}}    & = & (g/3 - h)/2.
\end{eqnarray}

Computing the sum of equations (\ref{e1}), (\ref{e4}) and (\ref{e6}), we see that 
\begin{eqnarray}
\label{e7}
4a_{\{0,0\}} + a_{\{0,1\}} + a_{\{0,2\}}  + a_{\{1,1\}} + a_{\{1,2\}} + a_{\{2,2\}} & = & 2(g/3 - h).
\end{eqnarray}
On the other hand, computing the sum of equations (\ref{e4}) and (\ref{e5}),
we obtain
\begin{eqnarray}
\label{e8}
a_{\{0,0\}} + a_{\{0,1\}} + a_{\{0,2\}}  + a_{\{1,1\}} + a_{\{1,2\}} + a_{\{2,2\}} & = & (g-h)/2.
\end{eqnarray}
Subtracting (\ref{e8}) from (\ref{e7}) and substituting $g = 3ht$, we have
\begin{eqnarray}
\label{e9}
3a_{\{0,0\}} & = & h(t-3)/2.
\end{eqnarray}
Therefore $h(t-3)/2$ is divisible by $3$ and hence $ht \equiv 0 \bmod 3$.
\end{proof}

When $h=1$, we obtain the following consequence of Theorem \ref{three.thm}. We note that this result is weaker than Corollary \ref{h=1.thm}.
\begin{corollary}
There does not exist a cyclic skew starter in $\zed_{3t}$ if $t \not\equiv 0 \bmod 3$.
\end{corollary}

A similar result can be proven when when $h=2$ or $4$. This result is new.
\begin{corollary}
\label{h=2.cor}
Suppose $t \not\equiv 0 \bmod 3$. Then there does not exist a cyclic skew frame starter of type 
$2^{3t}$ (in $\zed_{6t}$) or one of type   
$4^{3t}$ (in $\zed_{12t}$).
\end{corollary}

Here is another result that can be proved in a similar fashion.

\begin{theorem}
\label{four.thm}
There does not exist a cyclic skew strong starter of type 
$h^{4t}$ in $\zed_{4ht}$ if $ht \not\equiv 0 \bmod 4$.
\end{theorem}

\begin{proof}
Denote $g = 4ht$. 
Let $\phi : \zed_{g} \rightarrow \zed_4$ be the homomorphism defined as $x \mapsto x \bmod 4$.
Suppose $S = \{ \{x_i,y_i \} : 1 \leq i \leq (g-h)/2\}$ is a skew frame starter in $G \setminus H$, where $G = \zed_g$ and $H$ is the subgroup of $G$ having order $h$. 
Define the \emph{type} of an element $x \in G$ to be $\phi(x)$, and define the type of a pair $\{x,y\} \in S$ to be the multiset
$\{ \phi(x), \phi(y)\}$. For $ 0 \leq i \leq j \leq 3$, 
let $a_{\{i,j\}}$ denote the number of pairs in $S$ of type $\{i,j\}$.

$G \setminus H$ contains $g/4 - h$ elements of type $0$ and $g/4$ elements each of types $1$, $2$ and $3$. 
Thus the following four equations are obtained:
\begin{eqnarray}
\label{e41}2a_{\{0,0\}} + a_{\{0,1\}} + a_{\{0,2\}} + a_{\{0,3\}}  & = & g/4 - h\\
\label{e42}2a_{\{1,1\}} + a_{\{0,1\}} + a_{\{1,2\}}  + a_{\{1,3\}}  & = & g/4\\
\label{e43}2a_{\{2,2\}} + a_{\{0,2\}} + a_{\{1,2\}}   + a_{\{2,3\}}  & = & g/4\\
\label{e44}2a_{\{3,3\}} + a_{\{0,3\}} + a_{\{1,3\}}   + a_{\{2,3\}}  & = & g/4.
\end{eqnarray}

There are $\frac{1}{2}(\frac{g}{4} - h)$ pairs in $S$ having a difference of type $0$, 
$g/8$ pairs in $S$ that have a difference of type $2$, and $g/4$ pairs in $S$ that have a difference of type $1$ or $3$.
So we obtain three further equations:
\begin{eqnarray}
\label{e45} a_{\{0,0\}} + a_{\{1,1\}} + a_{\{2,2\}} + a_{\{23,3\}}  & = &  (g/4 - h)/2\\
\label{e46} a_{\{0,2\}} + a_{\{1,3\}}    & = &  g/8\\
\label{e47} a_{\{0,1\}} + a_{\{1,2\}} + a_{\{2,3\}} + a_{\{0,3\}}   & = & g/4.
\end{eqnarray}

Finally, there are $\frac{1}{2}(\frac{g}{4} - h)$ pairs in $S$ whose sum has type $0$ and
$g/8$ pairs in $S$ that have a sum of type $2$.
We obtain the following equations:
\begin{eqnarray}
\label{e48} a_{\{0,0\}} + a_{\{2,2\}}  + a_{\{1,3\}}   & = & (g/4 - h)/2\\
\label{e49} a_{\{0,2\}} + a_{\{1,1\}} + a_{\{3,3\}}   & = & g/8.
\end{eqnarray}

Computing the sum of equations (\ref{e41}) and (\ref{e43}), we see that 
\begin{eqnarray}
\label{e51}
2a_{\{0,0\}} + 2a_{\{2,2\}} + 2a_{\{0,2\}}  + a_{\{0,1\}} + a_{\{1,2\}} + a_{\{2,3\}} + a_{\{0,3\}} & = & g/2 - h.
\end{eqnarray}

Then, subtracting (\ref{e47}) from  (\ref{e51}) and dividing by 2, we obtain
\begin{eqnarray}
\label{e52}
a_{\{0,0\}} + a_{\{2,2\}} + a_{\{0,2\}}   & = & (g/4 - h)/2.
\end{eqnarray}

On the other hand, computing the sum of equations (\ref{e48}) and (\ref{e49}),
we obtain
\begin{eqnarray}
\label{e53}
a_{\{0,0\}} + a_{\{0,2\}} + a_{\{2,2\}}  + a_{\{1,1\}} + a_{\{1,3\}} + a_{\{3,3\}} & = & g/4 - h/2.
\end{eqnarray}

Subtracting (\ref{e49}) from  (\ref{e46}), we obtain
\begin{eqnarray}
\label{e50}
a_{\{1,1\}} +  a_{\{3,3\}} & = & a_{\{1,3\}}.
\end{eqnarray}

Substituting (\ref{e50}) into (\ref{e53}), we have
\begin{eqnarray}
\label{e54}
a_{\{0,0\}} + a_{\{0,2\}} + a_{\{2,2\}}   + 2a_{\{1,3\}}  & = & g/4 - h/2.
\end{eqnarray}

Subtracting (\ref{e52}) from  (\ref{e54}) and dividing by $2$, we have
\begin{eqnarray}
\label{e55}
a_{\{1,3\}}  & = & g/16.
\end{eqnarray}

Therefore $g = 4ht$ is divisible by $16$ and hence $ht \equiv 0 \bmod 4$.
\end{proof}

The following is an immediate corollary of Theorem \ref{four.thm}.

\begin{corollary}
\label{g=4ht.cor}
Suppose $t$ is odd. Then there does not exist a cyclic skew frame starter of type 
$2^{4t}$ in $\zed_{4t}$.
\end{corollary}
 
\section{Cyclic skew frame starters in  groups of small order}
\label{small.sec}

For $g = 21$, a cyclic skew frame starter of type $3^7$ is not ruled out by Theorem \ref{h=3.thm}. However, an exhaustive search shows that this skew frame starter does not exist. On the other hand, for $g = 39$  and $57$, we found cyclic skew frame starters of types $3^{13}$ and $3^{19}$, resp., by a backtracking algorithm. These starters are presented in Examples \ref{type3-13.exam} and \ref{type3-19.exam}.

\begin{example}
\label{type3-13.exam} 
A cyclic skew frame starter of type $3^{13}$ in $\zed_{39}$
\[
\begin{array}{l@{\quad}l@{\quad}l@{\quad}l@{\quad}l@{\quad}l@{\quad}l@{\quad}l}
\{29, 30\}  & \{32, 34\}  &  \{6,  9\} &  \{36,  1\}  & \{37,  3\}  & \{19, 25\}    \\
\{8, 15\}  & \{14, 22\} &  \{18, 27\}  & \{10, 20\}  &  \{5, 16\}  & \{38, 11   \\
\{21, 35\} &  \{28,  4\}  & \{17, 33\}  &  \{7, 24\} &  \{23,  2\}  & \{12, 31\}
\end{array}
\]
\end{example}

\begin{example}
\label{type3-19.exam} 
A cyclic skew frame starter of type $3^{19}$ in $\zed_{57}$
\[
\begin{array}{l@{\quad}l@{\quad}l@{\quad}l@{\quad}l@{\quad}l@{\quad}l@{\quad}l}
\{1  ,2\}   &   \{3 , 5\}    &  \{4 , 7\}   &   \{6 ,10\}   &   \{8 ,13 \}  &   \{9 ,15\}   &  \{11, 18\}    & \{17 ,25\}  \\   
\{26 ,35\} &   \{43 ,53\} &   \{45 ,56\} &   \{22 ,34\} &   \{27, 40\} &   \{30 ,44\} &   \{31 ,46\} &   \{23 ,39 \} \\  
\{33, 50\} &   \{37 ,55\} &   \{32, 52\} &   \{21 ,42\} &   \{29, 51\} &   \{24 ,47\} &   \{12 ,36\} &   \{48 ,16\} \\  
\{28 ,54\} &   \{14 ,41\} &   \{49 ,20\} 
\end{array}
\]
\end{example}

Theorem \ref{h=5.thm} does not rule out the existence of a cyclic skew frame starter of type $5^{5}$. However, from Theorem \ref{t^5strong.thm}, a strong frame starter of type $5^{5}$ in $\zed_{25}$ does not exist (so a cyclic skew starter also does not exist). 
Cyclic skew frame starters of types $5^{7}$ and $5^{11}$ were found using a backtracking algorithm.

\begin{example}
\label{type5-7.exam} 
A cyclic skew frame starter of type $5^7$ in $\zed_{35}$
\[
\begin{array}{l@{\quad}l@{\quad}l@{\quad}l@{\quad}l@{\quad}l@{\quad}l@{\quad}l}   
\{1 , 2\} &   \{29, 31\} &     \{8, 11\} &    \{18 ,22\} &    \{33 , 3\} &     \{6 ,12\} &    \{19 ,27\} &    \{16 ,25\}    \\
\{20 ,30\} &    \{34 ,10\} &    \{ 5, 17\} &    \{13, 26\} &     \{9 ,24\} &    \{23 , 4\} &    \{15 ,32\} 
\end{array}
\]
\end{example}

\begin{example}
\label{type5-11.exam} 
A cyclic skew frame starter of type $5^{11}$ in $\zed_{55}$
 \[
\begin{array}{l@{\quad}l@{\quad}l@{\quad}l@{\quad}l@{\quad}l@{\quad}l@{\quad}l}
 \{1 , 2\} &        \{3 , 5\} &        \{6 , 9\} &        \{4 , 8\} &        \{7 ,12\} &       \{10, 16\} &       \{14 ,21\} &       \{24 ,32\} \\      
  \{26 ,35\} &       \{19, 29\} &       \{41 ,53\} &       \{38 ,51\} &       \{34 ,48\} &       \{39 ,54\} &       \{31 ,47\} &       \{28, 45\} \\   
   \{25 ,43 \} &      \{23 ,42 \} &      \{30 ,50 \} &      \{15 ,36  \} &     \{17 ,40 \} &      \{13, 37  \} &     \{27, 52 \} &      \{49 ,20 \} \\  
    \{46 ,18\} 
\end{array}
\]
\end{example}  

We also have four small examples of cyclic skew frame starters, in  groups of even order, which were found by backtracking.

\begin{example}
\label{type4-5.exam} 
A cyclic skew frame starter of type $4^5$ in $\zed_{20}$
 \[
\begin{array}{l@{\quad}l@{\quad}l@{\quad}l@{\quad}l@{\quad}l@{\quad}l@{\quad}l}
\{8 , 9\} &    \{16 ,18\} &    \{14 ,17\} &    \{19  ,3\} &    \{ 1 , 7\} &    \{ 6 ,13\} &     \{4 ,12\} &   \{  2 ,11\}  
 \end{array}
\]
\end{example}

\begin{example}
\label{type8-5.exam} 
A cyclic skew frame starter of type $8^5$ in $\zed_{40}$
  \[
\begin{array}{l@{\quad}l@{\quad}l@{\quad}l@{\quad}l@{\quad}l@{\quad}l@{\quad}l}
 \{ 6 , 7\} &    \{27 ,29\} &    \{23 ,26\} &    \{12 ,16\} &  \{  33 ,39\} &    \{31 ,38\} &   \{ 14 ,22\} &   \{ 32 , 1 \} \\
   \{ 13 ,24\} &   \{ 37 , 9\} &     \{4 ,17\} &    \{34 , 8\} &    \{ 3 ,19\} &   \{ 11, 28\} &   \{ 18, 36\} &    \{ 2 ,21\} 
    \end{array}
\]
\end{example}

\begin{example}
\label{type2-25.exam} 
A cyclic skew frame starter of type $2^{25}$ in $\zed_{50}$
  \[
\begin{array}{l@{\quad}l@{\quad}l@{\quad}l@{\quad}l@{\quad}l@{\quad}l@{\quad}l}
    \{1 , 2\} &      \{3 , 5\} &      \{4 , 7\} &      \{6 ,10\} &     \{ 8 ,13\} &      \{9 ,15\} &     \{14 ,21\} &     \{20 ,28 \} \\  
    \{37 ,46\} &     \{38 ,48 \} &    \{23, 34\} &     \{24, 36\} &     \{18 ,31\} &     \{29 ,43 \} &    \{27, 42\} &     \{33, 49\} \\   
    \{30 ,47 \} &    \{26, 44\} &     \{22 ,41\} &     \{12 ,32\} &     \{19, 40\} &     \{45 ,17\} &     \{16, 39\} &     \{11, 35\} 
     \end{array}
\]
\end{example}

\begin{example}
\label{type4-13.exam} 
A cyclic skew frame starter of type $4^{13}$ in $\zed_{52}$
  \[
\begin{array}{l@{\quad}l@{\quad}l@{\quad}l@{\quad}l@{\quad}l@{\quad}l@{\quad}l}
 \{1 , 2\} &        \{3 , 5\} &        \{4 , 7\} &        \{6 ,10\} &        \{9 ,14\} &       \{11 ,17\} &       \{15 ,22\} &       \{27 ,35 \} \\  
\{24 ,33\} &       \{28 ,38\} &       \{31, 42\} &       \{29, 41\} &       \{20 ,34\} &       \{36 ,51\} &       \{21 ,37\} &       \{43,  8\} \\   
\{32, 50 \} &      \{30 ,49 \} &      \{44 ,12  \} &     \{19 ,40 \} &      \{25, 47 \} &      \{45 ,16 \} &      \{46 ,18 \} &      \{23 ,48\}   
\end{array}
\]
\end{example}

Table \ref{tab1} summarizes our current knowledge regarding ``small'' cyclic skew frame starters of type $t^u$  for $t > 1$. Note that the nonexistence of several cyclic skew frame starters follow from exhaustive backtracking searches, as indicated in the table. 

\begin{table}
\begin{center}
\begin{threeparttable}
\caption{Existence of small cyclic skew frame starters\tnote{1,2}}
\label{tab1}
\begin{tabular}{>{\centering}p{3.5cm}|>{\centering}p{3.5cm}|@{\hspace{1.5cm}}p{5cm}}
type & existence & \hspace{.75cm}authority \\ \hline
\rule[0mm]{0mm}{4mm}$2^5$ & yes & Theorem \ref{DS,SW.thm}\tnote{3} \\
$2^8$ & no & exhaustive search\\
$2^9$ & no & exhaustive search\\
$2^{12}$ & no & Corollary \ref{g=4ht.cor}\\
$2^{13}$ & yes & Theorem \ref{DS,SW.thm}\tnote{3}\\
$2^{16}$ & no & exhaustive search \\
$2^{17}$ & yes & Theorem \ref{DS,SW.thm}\tnote{3}\\
$2^{20}$ &  no & Corollary \ref{g=4ht.cor}  \\
$2^{21}$ & no & Corollary \ref{h=2.cor}\\
$2^{24}$ & no & Corollary \ref{h=2.cor}\\
$2^{25}$ &  yes & Example \ref{type2-25.exam} \\
$2^{28}$ & no & Corollary \ref{g=4ht.cor}  \\
$2^{29}$ &  yes & Theorem \ref{DS,SW.thm}\tnote{3} \\ \hline
\rule[0mm]{0mm}{4mm}$3^7$ & no & exhaustive search\\
$3^9$ & no & Corollary \ref{h=3.thm}\\
$3^{11}$ & no & Corollary \ref{h=3.thm}\\
$3^{13} $ & yes & Example \ref{type3-13.exam}\\ 
$3^{15} $ & no & Corollary \ref{h=3.thm} \\
$3^{17} $ & no & Corollary \ref{h=3.thm} \\
$3^{19} $ & yes & Example \ref{type3-19.exam}\\ 
 \hline
\rule[0mm]{0mm}{4mm}$4^5$ & yes & Example \ref{type4-5.exam}\\
$4^7$ & no & exhaustive search \\
$4^{8}$ & no & exhaustive search \\ 
$4^{9}$ &  no & exhaustive search  \\  
$4^{10}$ & no & exhaustive search   \\  
$4^{11}$ &  ? &   \\ \ 
$4^{12}$ &  no & Corollary \ref{h=2.cor}  \\ 
$4^{13}$ &  yes & Example \ref{type4-13.exam} \\\hline 
\rule[0mm]{0mm}{4mm}$5^7$ & yes & Example \ref{type5-7.exam}\\ 
$5^9$ & no & Corollary \ref{h=5.thm}\\ 
$5^{11}$ &  yes & Example \ref{type5-11.exam} \\ 
\hline
\rule[0mm]{0mm}{4mm}$6^5$ & no & exhaustive search \\ 
$6^8$ & ? &  \\ 
$6^9$ & ? &  \\ \hline
\rule[0mm]{0mm}{4mm}$8^5$ & yes &  Example \ref{type8-5.exam}\\
\end{tabular}
   \begin{tablenotes}
      \item[1] We only consider cyclic skew frame starters of type $t^u$  for $t > 1$ in this table. 
      \item[2] Theorems \ref{nonexist2.thm}, \ref{t^5strong.thm}, \ref{newnonexist.thm} and \ref{newnonexist2.thm} show that various cyclic frame starters of type $t^u$ do  not exist. These types $t^u$ are omitted from this table.
      \item[3] Theorem \ref{DS,SW.thm} yields a cyclic skew frame starter of type $2^q$ if $n=1$ and $q$ is prime.
    \end{tablenotes}
  \end{threeparttable}
  \end{center}
\end{table}

\section{Discussion and summary}
\label{summary.sec}

In my paper \cite{St22}, I investigated the existence of strong frame starters in cyclic groups, but I did not consider skew frame starters in that paper. Shortly after the completion of  \cite{St22}, Esther Lamken asked me if there is a skew frame starter of type $6^9$. I was not able to answer her question---either positively or negatively---but it motivated me to study cyclic skew frame starters in detail. 

There is an effective hill-climbing algorithm to find strong frame starters (see, e.g., \cite{DS92}). In conjunction with nonexistence results proven in \cite{St22} and elsewhere, it turned out that there was fairly compelling empirical evidence to support a conjecture I made in \cite{St22} regarding necessary and sufficient conditions for the existence of  strong frame starters in cyclic groups.

The situation is considerably more murky for skew frame starters in cyclic groups. First, it does not seem that there is a practical way to construct skew frame starters using hill-climbing. Therefore, the most effective search technique is backtracking, and exhaustive searches are impractical even in moderate-sized groups. The data presented in Table \ref{tab1} does not seem to me to be sufficient to propose any obvious general conjectures. However, for the orders considered, skew starters seem to be somewhat rare, as a number of exhaustive searches ended in failure. This suggests that there could be additional nonexistence results waiting to be proven, which is an interesting topic for future research. 

Finally, I note that results in noncyclic groups will differ from those obtained for cyclic groups. For example, there are 
no cyclic skew frame starters of types $4^4$, $2^9$ and $4^9$, but there are noncyclic skew frame starters of all these types. 

\section*{Acknowledgements}
Thanks to Jeff Dinitz for helpful discussions, and to Shannon Veitch for help with programming.

\end{document}